\documentclass[12pt]{article}
\usepackage{amssymb}
\usepackage{amsthm}
\usepackage{amsmath}
\usepackage{graphicx}
\usepackage{enumerate}
\usepackage{pict2e}

\DeclareMathOperator{\Max}{Max}
\DeclareMathOperator{\Min}{Min}

\newtheorem{theorem}{Theorem}[section]

\newtheorem{lemma}[theorem]{Lemma}
\newtheorem{proposition}[theorem]{Proposition}

\newtheorem{remark}[theorem]{Remark}
\newtheorem{example}[theorem]{Example}

\title{The Logic of Effect Algebras Incorporating Time Dimension}
\author{Ivan Chajda$^1$ and Helmut L\"anger$^2$}
\date{}

\begin{document}

\footnotetext[1]{Faculty of Science, Department of Algebra and Geometry, Palack\'y University Olomouc, 17.\ listopadu 12, 771 46 Olomouc, Czech Republic}

\footnotetext[2]{Faculty of Mathematics and Geoinformation, Institute of Discrete Mathematics and Geometry, TU Wien, Wiedner Hauptstra\ss e 8-10, 1040 Vienna, Austria, and Faculty of Science, Department of Algebra and Geometry, Palack\'y University Olomouc, 17.\ listopadu 12, 771 46 Olomouc, Czech Republic}

\footnotetext[3]{Support of the research of both authors by the Austrian Science Fund (FWF), project I~4579-N, entitled ``The many facets of orthomodularity'' and, concerning the first author, by IGA, project P\v rF~2023~010, is gratefully acknowledged.}

\footnotetext[4]{Helmut L\"anger (corresponding author), helmut.laenger@tuwien.ac.at}

\footnotetext[5]{Ivan Chajda, ivan.chajda@upol.cz}

\maketitle

\begin{abstract}
Effect algebras were introduced in order to describe the structure of effects, i.e.\ events in quantum mechanics. They are partial algebras describing the logic behind the corresponding events. It is natural to ask how to introduce the logical connective implication in effect algebras. For lattice-ordered effect algebras this task was already solved by several authors, including the present ones. We concentrate on effect algebras that need not be lattice-ordered since these can better describe the events occurring in the quantum physical system. Although an effect algebra is only partial, we try to find a logical connective implication which is everywhere defined. But such a connective can be ``unsharp'' or ``inexact'' because its outputs for given pairs of entries need not be elements of the underlying effect algebra, but may be subsets of (mutually incomparable) maximal elements. We introduce such an implication together with its adjoint functor representing conjunction. Then we consider so-called tense operators on effect algebras. Of course, also these operators turn out to be ``unsharp'' in the aforementioned sense, but they are in a certain relation with the operators implication and conjunction. Finally, for given tense operators and given time set $T$, we describe two methods how to construct a time preference relation $R$ on $T$ such that the given tense operators are either comparable with or equivalent to those induced by the time frame $(T,R)$.
\end{abstract}



\section{Introduction}

As explained by D.~J.~Foulis and M.~K.~Bennett \cite{FB} and by R.~Giuntini and H.~Greuling \cite{GG}, effect algebras were introduced in order to formalize effects of quantum mechanics. This process of formalization is described in these papers and in detail also in the monograph \cite{DP} by A.~Dvure\v censkij and S.~Pulmannov\'a and hence it is not necessary to repeat it here. However, there are still two aspects which were not investigated in these sources.

Namely, if one considers effect algebras as a formalization of the logic of quantum mechanics then the natural question arises what are the logical connectives derived from them. One of the aims of our paper is to analyze three possibilities how to define the logical connectives implication and conjunction such that these form adjoint pairs. The importance of this requirement is that when the adjoint pair is established than the corresponding logic is equipped with the derivation rule Modus Ponens.

The second aspect is that every physical system is dynamic, it means that the values of logical propositions on this system vary with respect to time. Hence, this is valid also for our logics derived from effect algebras. Thus one can introduce a time scale (usually called a time frame in the corresponding tense logic), i.e.\ a time set together with some time preference relation, and evaluate propositions in given time coordinates. As mentioned e.g.\ in \cite{CP15a}, so-called tense operators can be a suitable tool for such a time analysis. We will show how these tense operators can be constructed in the logic based on effect algebras and describe their properties. However, also conversely, when a time set and tense operators are given within the considered logic, we will show that a time preference relation can be constructed in such a way that the given tense operators are induced by this relation.

\section{Preliminaries}

Consider a poset $\mathbf P=(P,\leq)$ and let $A,B\subseteq P$. If it has a bottom element, this element will be denoted by $0$. If $\mathbf P$ has a top element, this element will be denoted by $1$. The {\em poset} $\mathbf P$ is called {\em bounded} if it has both $0$ and $1$, and in this case it will be denoted by $\mathbf P=(P,\leq,0,1)$.

By a {\em binary operator} on $P$ we understand a mapping from $P^2$ to $2^P$, i.e.\ it assigns to every pair $(x,y)$ of elements of $P$ a subset of $P$. In what follows, for the sake of brevity, we will not distinguish between a singleton $\{a\}$ and its unique element $a$.

The poset $\mathbf P$ is said to satisfy the
\begin{itemize}
\item Ascending Chain Condition (shortly ACC) if there are no infinite ascending chains in $\mathbf P$,
\item Descending Chain Condition (shortly DCC) if there are no infinite descending chains in $\mathbf P$.
\end{itemize}
Notice that every finite poset satisfies both the ACC and the DCC. Further, let $\Max A$ and $\Min A$ denote the set of all maximal and minimal elements of $A$, respectively. If $\mathbf P$ satisfies the ACC then for every $a\in A$ there exists some $b\in\Max A$ with $a\leq b$. This implies that if $A$ is not empty the same is true for $\Max A$. The corresponding assertion holds for the DCC and $\Min A$.

We define
\begin{align*}
      A\leq_1B & \text{ if for every }a\in A\text{ there exists some }b\in B\text{ with }a\leq b, \\
      A\leq_2B & \text{ if for every }b\in B\text{ there exists some }a\in A\text{ with }a\leq b, \\
A\sqsubseteq B & \text{ if there exists some }a\in A\text{ and some }b\in B\text{ with }a\leq b, \\
   A\approx_1B & \text{ if both }A\leq_1B\text{ and }B\leq_1A, \\
   A\approx_2B & \text{ if both }A\leq_2B\text{ and }B\leq_2A.
\end{align*}
For every set $A$ we denote the set of its non-empty subsets by $\mathcal P_+A$.

\section{Effect algebras}

The concept of an effect algebra was introduced in 1989 by R.~Giuntini and H.~Greuling \cite{GG} under a different name. The name effect algebra was used the first time by D.~J.~Foulis and M.~K.~Bennett \cite{FB}, see e.g.\ also \cite{DP}.

Recall from \cite{DP} that an {\em effect algebra} is a partial algebra $(E,+,0,1)$ of type $(2,0,0)$ satisfying the following conditions:
\begin{enumerate}[(E1)]
\item If $a,b\in E$ and $a+b$ is defined then so is $b+a$ and both coincide,
\item if $a,b,c\in E$ and $a+b$ and $(a+b)+c$ are defined then so are $b+c$ and $a+(b+c)$ and $(a+b)+c=a+(b+c)$,
\item for each $a\in E$ there exists a unique $b\in E$ with $a+b=1$; in the sequel, this element $b$ will be denoted by $a'$ and called the {\em supplement} of $a$,
\item $a+1$ is defined only if $a=0$.
\end{enumerate}
Because of (E3), $'$ is a unary operation on $E$ and we will write effect algebras in the form $(E,+,{}',0,1)$.

In the following let $\mathbf E=(E,+,{}',0,1)$ be an effect algebra, $a,b\in E$ and $A,B\subseteq E$.

On $E$ we introduce a binary relation $\leq$ as follows:
\[
a\leq b\text{ if there exists some }c\in E\text{ with }a+c=b.
\]
As shown e.g.\ in \cite{DP}, $(E,\leq,0,1)$ is a bounded poset. If $(E,\leq)$ is even a lattice then $\mathbf E$ is called a {\em lattice effect algebra}.

\begin{example}\label{ex1}
If $E:=\{0,a,b,c,d,a',b',c',1\}$ and a partial binary operation $+$ and a unary operation ${}'$ on $E$ are defined by:
\[
\begin{array}{l|lllllllll}
+  & 0  & a  & b  & c  & d  & c' & b' & a' & 1 \\
\hline
0  & 0  & a  & b  & c  & d  & c' & b' & a' & 1 \\
a  & a  & -  & c' & b' & -  & -  & -  & 1  & - \\
b  & b  & c' & d  & a' & b' & -  & 1  & -  & - \\
c  & c  & b' & a' & -  & -  & 1  & -  & -  & - \\
d  & d  & -  & b' & -  & 1  & -  & -  & -  & - \\
c' & c' & -  & -  & 1  & -  & -  & -  & -  & - \\
b' & b' & -  & 1  & -  & -  & -  & -  & -  & - \\
a' & a' & 1  & -  & -  & -  & -  & -  & -  & - \\
1  & 1  & -  & -  & -  & -  & -  & -  & -  & -
\end{array}
\quad\quad\quad
\begin{array}{l|l}
x  & x' \\
\hline
0  & 1 \\
a  & a' \\
b  & b' \\
c  & c' \\
d  & d \\
c' & c \\
b' & b \\
a' & a \\
1' & 0
\end{array}
\]
then $(E,+,{}',0,1)$ is a non-lattice effect algebra whose induced poset is depicted in Figure~1:

\vspace*{-2mm}

\[
\setlength{\unitlength}{7mm}
\begin{picture}(6,9)
\put(3,2){\circle*{.3}}
\put(1,4){\circle*{.3}}
\put(3,4){\circle*{.3}}
\put(5,4){\circle*{.3}}
\put(3,5){\circle*{.3}}
\put(1,6){\circle*{.3}}
\put(3,6){\circle*{.3}}
\put(5,6){\circle*{.3}}
\put(3,8){\circle*{.3}}
\put(3,2){\line(-1,1)2}
\put(3,2){\line(0,1)6}
\put(3,2){\line(1,1)2}
\put(1,6){\line(0,-1)2}
\put(1,6){\line(1,-1)2}
\put(1,6){\line(1,1)2}
\put(5,6){\line(0,-1)2}
\put(5,6){\line(-1,-1)2}
\put(5,6){\line(-1,1)2}
\put(3,6){\line(-1,-1)2}
\put(3,6){\line(1,-1)2}
\put(2.875,1.25){$0$}
\put(.35,3.85){$a$}
\put(3.4,3.85){$b$}
\put(5.4,3.85){$c$}
\put(3.4,4.85){$d$}
\put(.35,5.85){$c'$}
\put(3.4,5.85){$b'$}
\put(5.4,5.85){$a'$}
\put(2.85,8.4){$1$}
\put(2.3,0){{\rm Fig.~1}}
\put(-.4,-1.1){{\rm A non-lattice effect algebra}}
\end{picture}
\]
\end{example}              

\vspace*{6mm}

The elements $a$ and $b$ are called {\em orthogonal} to each other (shortly $a\perp b$) if $a\leq b'$. It can be shown that $a+b$ is defined if and only if $a\perp b$. We define a partial binary operation $\odot$ on $E$ by $a\odot b:=(a'+b')'$. It is evident that $a\odot b$ is defined if and only if $a'\perp b'$.

\begin{example}
The operation table of $\odot$ corresponding to the effect algebra from Example~\ref{ex1} looks as follows:
\[
\begin{array}{l|lllllllll}
\odot & 0 & a & b & c & d & c' & b' & a' & 1 \\
\hline
0     & - & - & - & - & - & -  & -  & -  & 0 \\
a     & - & - & - & - & - & -  & -  & 0  & a \\
b     & - & - & - & - & - & -  & 0  & -  & b \\
c     & - & - & - & - & - & 0  & -  & -  & c \\
d     & - & - & - & - & 0 & -  & b  & -  & d \\
c'    & - & - & - & 0 & - & -  & a  & b  & c' \\
b'    & - & - & 0 & - & b & a  & d  & c  & b' \\
a'    & - & 0 & - & - & - & b  & c  & -  & a' \\
1     & 0 & a & b & c & d & c' & b' & a' & 1
\end{array}
\]
\end{example}

We say that
\begin{itemize}
\item $A+B$ is defined if so is $a+b$ for all $a\in A$ and all $b\in B$,
\item $A\odot B$ is defined if so is $a\odot b$ for all $a\in A$ and all $b\in B$,
\item $A\leq B$ if $a\leq b$ for all $a\in A$ and all $b\in B$.
\end{itemize}
If $A+B$ is defined we put $A+B:=\{a+b\mid a\in A,b\in B\}$. If $A\odot B$ is defined we put $A\odot B:=\{a\odot b\mid a\in A,b\in B\}$.

The following result is well-known, see e.g.\ \cite{DP} or \cite{FB}.

\begin{lemma}\label{lem3}
If $(E,+,{}',0,1)$ is an effect algebra and $a,b,c\in E$ then
\begin{enumerate}[{\rm(i)}]
\item $(E,\leq,{}',0,1)$ is a bounded poset with an antitone involution,
\item $a+0=a\odot1=a$,
\item $a+a'=1$ and $a\odot a'=0$,
\item $a,b\leq a+b$ and $a\odot b\leq a,b$,
\item if $a\leq b$ then $a=b\odot(a+b')=\big(b'+(a+b')'\big)'$,
\item if $a\leq b$ then $b=a+(a'\odot b)=a+(a+b')'$,
\item if $a\leq b$ and $b+c$ is defined so is $a+c$ and we have $a+c\leq b+c$,
\item if $a\leq b$ and $a\odot c$ is defined so is $b\odot c$ and we have $a\odot c\leq b\odot c$,
\item if $a+b$ and $a+c$ are defined and $a+b=a+c$ then $b=c$,
\item if $a\odot b$ and $a\odot c$ are defined and $a\odot b=a\odot c$ then $b=c$.
\end{enumerate}
\end{lemma}

\section{Logical connectives in effect algebras}

As stated in \cite{FB}, effect algebras serve as a model for unsharp quantum logic. Hence there is the question, what are the logical connectives within this logic. Usually, the partial operations $+$ and $\odot$ are considered as disjunction and conjunction, respectively. In every logic the most productive connective, however, is implication. Using this connective it is possible to derive new propositions from given ones by certain derivation rules (e.g.\ Modus Ponens or substitution rule). The question arises how to introduce the connective implication in the quantum logic based on an effect algebra. For lattice-ordered effect algebras this problem was already solved by the authors in \cite{CL22}. Now we will investigate effect algebras that need not be lattice-ordered. It is worth noticing that the present authors together with R.~Hala\v s derived a Gentzen system for the connective implication in lattice effect algebras and also for the non-lattice case, but the implication treated there differs essentially from that we will investigate now. The main difference is that now we are going to connect our implication with conjunction via a certain kind of adjointness.

Let $\mathbf E=(E,+,{}',0,1)$ be an effect algebra. Define the following binary operator on $E$:
\[
b\rightarrow c:=\Max\{x\in E\mid x\odot b\text{ is defined and }x\odot b\leq c\}
\]
($b,c\in E$). This is our ``unsharp'' implication in the logic based on $\mathbf E$. The denotation ``unsharp'' expresses the fact that the result of $b\rightarrow c$ need not be an element of $E$ (as it was the case for the implication introduced in \cite{CHL} and \cite{CL22}), but may be a subset of $E$. The elements of $b\rightarrow c$ form an antichain, it means we cannot prefer one with respect to another by their order. Moreover, $b\rightarrow c$ is defined for all $b,c\in E$.

\begin{example}
The ``operation table'' of $\rightarrow$ corresponding to the effect algebra from Example~\ref{ex1} looks as follows:
\[
\begin{array}{l|lllllllll}
\rightarrow & 0  & a  & b  & c  & d  & c'       & b'        & a'       & 1 \\
\hline
0           & 1  & 1  & 1  & 1  & 1  & 1        & 1         & 1        & 1 \\
a           & a' & 1  & a' & a' & a' & 1        & 1         & a'       & 1 \\
b           & b' & b' & 1  & b' & 1  & 1        & 1         & 1        & 1 \\
c           & c' & c' & c' & 1  & c' & c'       & 1         & 1        & 1 \\
d           & d  & d  & b' & d  & 1  & b'       & 1         & b'       & 1 \\
c'          & c  & b' & a' & c  & a' & 1        & \{a',b'\} & a'       & 1 \\
b'          & b  & c' & d  & a' & b' & \{d,c'\} & 1         & \{d,a'\} & 1 \\
a'          & a  & a  & c' & b' & c' & c'       & \{b',c'\} & 1        & 1 \\
1           & 0  & a  & b  & c  & d  & c'       & b'        & a'       & 1
\end{array}
\]
\end{example}

The next result shows that our unsharp implication still shares some important properties asked usually in any non-classical logic.

\begin{lemma}\label{lem1}
Let $(E,+,{}',0,1)$ be an effect algebra satisfying the {\rm ACC} and $a,b,c\in E$. Then the following holds:
\begin{enumerate}[{\rm(i)}]
\item $a'\leq_1a\rightarrow b$ and $a\rightarrow b\neq\emptyset$,
\item $a\rightarrow0=a'$ and $1\rightarrow a=a$,
\item $a\rightarrow b=1$ if and only if $a\leq b$,
\item if $a\leq b$ then $c\rightarrow a\leq_1c\rightarrow b$.
\end{enumerate}
\end{lemma}

\begin{proof}
Put $A:=\{x\in E\mid x\odot a\text{ is defined and }x\odot a\leq b\}$.
\begin{enumerate}[(i)]
\item $a'\leq_1a\rightarrow b$ follows from $a'\in A$ and from the ACC. $a\rightarrow b\neq\emptyset$ follows from $a'\leq_1a\rightarrow b$.
\item We have
\begin{align*}
 a\rightarrow0 & =\Max\{x\in E\mid x\odot a\text{ is defined and }x\odot a\leq0\}= \\
               & =\Max\{x\in E\mid a'\leq x\text{ and }x'+a'=1\}= \\
	  					 & =\Max\{x\in E\mid a'\leq x\text{ and }x=a'\}=a', \\
1\rightarrow a & =\Max\{x\in E\mid x\odot1\text{ is defined and }x\odot1\leq a\}=a.
\end{align*}
\item The following are equivalent:
\begin{align*}
& a\rightarrow b=1, \\
& \Max A=1, \\
& 1\in A, \\
& 1\odot a\text{ is defined and }1\odot a\leq b, \\
& a\leq b.
\end{align*}
\item If $a\leq b$ then
\[
\{x\in E\mid x\odot c\text{ is defined and }x\odot c\leq a\}\subseteq\{x\in E\mid x\odot c\text{ is defined and }x\odot c\leq b\}
\]
whence $c\rightarrow a\leq_1c\rightarrow b$.
\end{enumerate}
\end{proof}

At first we show some kind of adjointness between $\odot$ and $\rightarrow$.

\begin{theorem}\label{th1}
Let $(E,+,{}',0,1)$ be an effect algebra satisfying the {\rm ACC} and $a,b,c\in E$ and assume $a\odot b$ to be defined. Then
\[
a\odot b\leq c\text{ if and only if }a\leq_1b\rightarrow c.
\]
\end{theorem}

\begin{proof}
If $a\odot b\leq c$ then $a\in\{x\in E\mid x\odot b\text{ is defined and }x\odot b\leq c\}$ and hence $a\leq_1b\rightarrow c$. Conversely, if $a\leq_1b\rightarrow c$ then there exists some $d\in b\rightarrow c$ with $a\leq d$ and we have that $d\odot b$ is defined and $d\odot b\leq c$ and hence $a\odot b\leq d\odot b\leq c$ according to Lemma~\ref{lem3}.
\end{proof}

Adjointness shown in Theorem~\ref{th1} yields an important derivation rule valid in this logic. Namely,
\[
a\rightarrow b\leq_1a\rightarrow b
\]
implies by adjointness
\[
a\odot(a\rightarrow b)=(a\rightarrow b)\odot a\leq b
\]
(provided $a\odot(a\rightarrow b)$ is defined) saying properly that the value of $b$ cannot be less than the value of the conjunction of $a$ and $a\rightarrow b$ (provided this conjunction is defined), which is just the derivation rule Modus Ponens.

\begin{lemma}\label{lem2}
Let $(E,+,{}',0,1)$ be an effect algebra satisfying the {\rm ACC} and $a,b\in E$. Then the following holds:
\begin{enumerate}[{\rm(i)}]
\item $(a\rightarrow b)\odot a$ is defined and $(a\rightarrow b)\odot a\leq b$.
\item If $a\odot b$ is defined then $a\leq_1b\rightarrow(a\odot b)$.
\end{enumerate}
\end{lemma}

\begin{proof}
\
\begin{enumerate}[(i)]
\item According to the definition of $a\rightarrow b$, for all $x\in a\rightarrow b$ we have that $x\odot a$ is defined and $x\odot a\leq b$.
\item If $a\odot b$ is defined then $a\in\{x\in E\mid x\odot b\text{ is defined and }x\odot b\leq a\odot b\}$ and hence $a\leq_1b\rightarrow(a\odot b)$.
\end{enumerate}
\end{proof}

Hence, if $(a\rightarrow b)\odot a$ is defined, evidently $(a\rightarrow b)\odot a\leq a$, thus, together with Lemma~\ref{lem2}, we obtain
\[
(a\rightarrow b)\odot a\leq_1\Max L(a,b)
\]
which is {\em divisibility}.

There is also another possibility how to define the connective implication in an effect algebra $(E,+,{}',0,1)$, namely as the following partial binary operation:
\[
b\rightsquigarrow c:=b'+c
\]
($b,c\in E$). It is evident that $b\rightsquigarrow c$ is defined if and only if $c\leq b$.

\begin{example}
The operation table of $\rightsquigarrow$ corresponding to the effect algebra from Example~\ref{ex1} looks as follows:
\[
\begin{array}{l|lllllllll}
\rightsquigarrow & 0  & a  & b  & c  & d  & c' & b' & a' & 1 \\
\hline
0                & 1  & -  & -  & -  & -  & -  & -  & -  & - \\
a                & a' & 1  & -  & -  & -  & -  & -  & -  & - \\
b                & b' & -  & 1  & -  & -  & -  & -  & -  & - \\
c                & c' & -  & -  & 1  & -  & -  & -  & -  & - \\
d                & d  & -  & b' & -  & 1  & -  & -  & -  & - \\
c'               & c  & b' & a' & -  & -  & 1  & -  & -  & - \\
b'               & b  & c' & d  & a' & b' & -  & 1  & -  & - \\
a'               & a  & -  & c' & b' & -  & -  & -  & 1  & - \\
1                & 0  & a  & b  & c  & d  & c' & b' & a' & 1
\end{array}
\]
\end{example}

Also this implication which is a partial operation satisfies several important properties of implication known in non-classical logics.

\begin{lemma}\label{lem4}
Let $(E,+,{}',0,1)$ be an effect algebra and $a,b,c\in E$. Then the following holds:
\begin{enumerate}[{\rm(i)}]
\item if $a\rightsquigarrow b$ is defined then $a'\leq_1a\rightsquigarrow b$,
\item $a\rightsquigarrow0=a'$ and $1\rightsquigarrow a=a'$,
\item $a\rightsquigarrow b=1$ if and only if $a=b$,
\item if $a\leq b$ and $c\rightsquigarrow b$ is defined then $c\rightsquigarrow a$ is defined and $c\rightsquigarrow a\leq c\rightsquigarrow b$.
\end{enumerate}
\end{lemma}

\begin{proof}
\
\begin{enumerate}[(i)]
\item If $a\rightsquigarrow b$ is defined then $a'\leq a'+b=a\rightsquigarrow b$.
\item We have $a\rightsquigarrow0=a'+0=a'$ and $1\rightsquigarrow a=1'+a=a$.
\item The following are equivalent: $a\rightsquigarrow b=1$; $a'+b=1$; $a=b$.
\item If $a\leq b$ and $c\rightsquigarrow b$ is defined then $c'+b$ and hence, according to Lemma~\ref{lem3}, $c'+a$, i.e.\ $c\rightsquigarrow a$, is defined and
\[
c\rightsquigarrow a=c'+a\leq c'+b=c\rightsquigarrow b.
\]
\end{enumerate}
\end{proof}

The following theorem shows the relationship between $\rightarrow$ and $\rightsquigarrow$.

\begin{theorem}\label{th2}
Let $(E,+,{}',0,1)$ be an effect algebra and $a,b\in E$ and assume $a\rightsquigarrow b$ to be defined. Then $a\rightarrow b=a\rightsquigarrow b$.
\end{theorem}

\begin{proof}
We have $b\leq a$ and $(a\rightsquigarrow b)\odot a=b$ according to Lemma~\ref{lem3}. If $x\in E$, $x\odot a$ is defined and $x\odot a\leq b$ then
\[
x=a'+(x\odot a)\leq a'+b=a\rightsquigarrow b
\]
according to Lemma~\ref{lem3}. Hence $a\rightsquigarrow b=a\rightarrow b$.
\end{proof}

Also for the implication $\rightsquigarrow$ we can show some kind of adjointness which is, however, only partial.

\begin{theorem}
Let $(E,+,{}',0,1)$ be an effect algebra and $a,b,c\in E$ and assume $a\odot b$ and $b\rightsquigarrow c$ to be defined. Then
\[
a\odot b\leq c\text{ if and only if }a\leq b\rightsquigarrow c.
\]
\end{theorem}

\begin{proof}
If $a\odot b\leq c$ then $a=b'+(a\odot b)\leq b'+c=b\rightsquigarrow c$ according to Lemma~\ref{lem3}. If, conversely, $a\leq b\rightsquigarrow c$ then $a\odot b\leq(b\rightsquigarrow c)\odot b=(b'+c)\odot b=c$ according to Lemma~\ref{lem3}.
\end{proof}

There is a further possibility how to define the connective implication in an effect algebra $\mathbf E=(E,+,{}',0,1)$ that need not be lattice-ordered, namely
\[
b\Rightarrow c:=b'+\Max L(b,c)
\]
($b,c\in E$, cf.\ \cite{CHL}). If $\mathbf E$ satisfies the ACC then $b\Rightarrow c\neq\emptyset$. This kind of implication is, of course, again an unsharp one since the result of $a\Rightarrow b$ need not be an element of the corresponding effect algebra, but may be a subset of it. On the other hand, we see that $\Rightarrow$ is an everywhere defined binary operator on $E$.

\begin{example}
The ``operation table'' of $\Rightarrow$ corresponding to the effect algebra from Example~\ref{ex1} looks as follows:
\[
\begin{array}{l|lllllllll}
\Rightarrow & 0  & a  & b  & c  & d  & c'       & b'        & a'       & 1 \\
\hline
0           & 1  & 1  & 1  & 1  & 1  & 1        & 1         & 1        & 1 \\
a           & a' & 1  & a' & a' & a' & 1        & 1         & a'       & 1 \\
b           & b' & b' & 1  & b' & 1  & 1        & 1         & 1        & 1 \\
c           & c' & c' & c' & 1  & c' & c'       & 1         & 1        & 1 \\
d           & d  & d  & b' & d  & 1  & b'       & 1         & b'       & 1 \\
c'          & c  & b' & a' & c  & a' & 1        & \{a',b'\} & a'       & 1 \\
b'          & b  & c' & d  & a' & b' & \{d,c'\} & 1         & \{d,a'\} & 1 \\
a'          & a  & a  & c' & b' & c' & c'       & \{b',c'\} & 1        & 1 \\
1           & 0  & a  & b  & c  & d  & c'       & b'        & a'       & 1
\end{array}
\]
\end{example}

The properties of the implication $\Rightarrow$ are very natural, see the following result.

\begin{lemma}\label{lem5}
Let $(E,+,{}',0,1)$ be an effect algebra satisfying the {\rm ACC} and $a,b,c\in E$. Then the following holds:
\begin{enumerate}[{\rm(i)}]
\item $a'\leq a\Rightarrow b\neq\emptyset$,
\item $a\Rightarrow0=a'$ and $1\Rightarrow a=a$,
\item $a\Rightarrow b=1$ if and only if $a\leq b$,
\item if $a\leq b$ then $c\Rightarrow a\leq_1c\Rightarrow b$.
\end{enumerate}
\end{lemma}

\begin{proof}
\
\begin{enumerate}[(i)]
\item $a'\leq_1a\Rightarrow b$ follows from the definition of $a\Rightarrow b$ according to Lemma~\ref{lem3}. $a\Rightarrow b\neq\emptyset$ follows from the ACC.
\item We have $a\Rightarrow0=a'+\Max L(a,0)=a'$ and $1\Rightarrow a=1'+\Max L(1,a)=a$.
\item The following are equivalent: $a\Rightarrow b=1$; $a'+\Max L(a,b)=1$; $\Max L(a,b)=a$; $a\leq b$.
\item If $a\leq b$ then $L(c,a)\subseteq L(c,b)$ and hence $\Max L(c,a)\leq_1\Max L(c,b)$ whence
\[
c\Rightarrow a=c'+\Max L(c,a)\leq_1c'+\Max L(c,b)=c\Rightarrow b
\]
according to Lemma~\ref{lem3}.
\end{enumerate}
\end{proof}

Now we can compare all three implications considered here.

\begin{theorem}
Let $(E,+,{}',0,1)$ be an effect algebra and $a,b\in E$. Then the following holds:
\begin{enumerate}[{\rm(i)}]
\item $a\Rightarrow b\leq_1a\rightarrow b$.
\item If $a\rightsquigarrow b$ is defined then $a\rightarrow b=a\Rightarrow b=a\rightsquigarrow b$,
\end{enumerate}
\end{theorem}

\begin{proof}
\
\begin{enumerate}[(i)]
\item Let $c\in a\Rightarrow b$. Then there exists some $d\in\Max L(a,b)$ with $a'+d=c$. Now $c\odot a$ is defined and according to Lemma~\ref{lem3} we have $c\odot a=(a'+d)\odot a=d\leq b$ and hence $c\leq_1a\rightarrow b$.
\item Assume $a\rightsquigarrow b$ to be defined. That $a\rightarrow b=a\rightsquigarrow b$ was proven in Theorem~\ref{th2}. Because of $b\leq a$ we have
\[
a\Rightarrow b=a'+\Max L(a,b)=a'+b=a\rightsquigarrow b.
\]
\end{enumerate}
\end{proof}

There is the question how to define a binary operator $\otimes$ on an effect algebra $\mathbf E=(E,+,{}',0,$ $1)$ such that $\otimes$ and $\Rightarrow$ form an adjoint pair. For this purpose we define
\[
a\otimes b:=\Min U(a,b')\odot b
\]
($a,b\in E$). It is easy to see that
\begin{align*}
    a\otimes b & =(b\Rightarrow a')', \\
a\Rightarrow b & =(b'\otimes a)'
\end{align*}
for all $a,b\in E$. For subsets $A,B$ of $E$ we define
\begin{align*}
    A\otimes B & :=\{a\otimes b\mid a\in A,b\in B\}, \\
A\Rightarrow B & :=\{a\Rightarrow b\mid a\in A,b\in B\}.
\end{align*}

\begin{example}
The ``operation table'' of $\otimes$ corresponding to the effect algebra from Example~\ref{ex1} looks as follows:
\[
\begin{array}{l|lllllllll}
\otimes & 0 & a & b & c  & d & c'     & b'       & a'      & 1 \\
\hline
0       & 0 & 0 & 0 & 0 & 0 & 0       & 0        & 0       & 0 \\
a       & 0 & a & 0 & 0 & b & a       & \{d,a'\} & 0       & a \\
b       & 0 & 0 & 0 & 0 & 0 & \{a,b\} & 0        & \{b,c\} & b \\
c       & 0 & 0 & 0 & c & b & 0       & \{c,d\}  & c       & c \\
d       & 0 & a & 0 & c & 0 & a       & b        & c       & d \\
c'      & 0 & a & b & 0 & d & c'      & a        & b       & c' \\
b'      & 0 & a & 0 & c & b & a       & d        & c       & b' \\
a'      & 0 & 0 & b & c & d & b       & c        & a'      & a' \\
1       & 0 & a & b & c & d & c'      & b'       & a'      & 1
\end{array}
\]
\end{example}

Analogous to Lemma~\ref{lem5} we obtain

\begin{lemma}
Let $(E,+,{}',0,1)$ be an effect algebra satisfying the {\rm DCC} and $a,b,c\in E$. Then the following holds:
\begin{enumerate}[{\rm(i)}]
\item $\emptyset\neq a\otimes b\leq b$,
\item $a\otimes1=1\otimes a=a$,
\item $a\otimes b=0$ if and only if $a\perp b$,
\item if $a\leq b$ then $a\otimes c\leq_2b\otimes c$.
\end{enumerate}
\end{lemma}

We can show that also the operators $\otimes$ and $\Rightarrow$ form an adjoint pair.

\begin{theorem}
Let $(E,+,{}',0,1)$ be an effect algebra satisfying both the {\rm ACC} and the {\rm DCC} and $a,b,c\in E$. Then the following holds:
\begin{enumerate}[{\rm(i)}]
\item $a\otimes b\sqsubseteq c$ if and only if $a\sqsubseteq b\Rightarrow c$ {\rm(}{\em adjointness}{\rm)},
\item $(a\Rightarrow b)\otimes a=\Max L(a,b)$ {\rm(}{\em divisibility}{\rm)}.
\end{enumerate}
\end{theorem}

\begin{proof}
\
\begin{enumerate}[(i)]
\item According to Lemma~\ref{lem3} any of the following statements implies the next one:
\begin{align*}
                                                           a\otimes b & \sqsubseteq c, \\
                                                  \Min U(a,b')\odot b & \sqsubseteq c, \\
       \text{there exists some }d\in\Min U(a,b')\odot b\text{ with }d & \leq c, \\
       \text{there exists some }d\in\Min U(a,b')\odot b\text{ with }d & \in L(b,c), \\
       \text{there exists some }d\in\Min U(a,b')\odot b\text{ with }d & \sqsubseteq\Max L(b,c), \\
                                                  \Min U(a,b')\odot b & \sqsubseteq \Max L(b,c), \\
                                    b'+\big(\Min U(a,b')\wedge b\big) & \sqsubseteq b'+\Max L(b,c), \\
                                                    a\leq\Min U(a,b') & \sqsubseteq b\Rightarrow c, \\
                                                                    a & \sqsubseteq b\Rightarrow c, \\
                                                                    a & \sqsubseteq b'+\Max L(b,c), \\
           \text{there exists some }e\in b'+\Max L(b,c)\text{ with }a & \leq e, \\   
           \text{there exists some }e\in b'+\Max L(b,c)\text{ with }e & \in U(a,b'), \\   
\text{there exists some }e\in b'+\Max L(b,c)\text{ with }\Min U(a,b') & \sqsubseteq e, \\   
                                                         \Min U(a,b') & \sqsubseteq b'+\Max L(b,c), \\
                                                  \Min U(a,b')\odot b & \sqsubseteq\big(b'+\Max L(b,c)\big)\odot b, \\
                                                           a\otimes b & \sqsubseteq\Max L(b,c)\leq c, \\
                                                           a\otimes b & \sqsubseteq c.
\end{align*}
\item Since $\Max L(a,b)\leq a$ we can apply (v) of Lemma~\ref{lem3} in order to compute
\begin{align*}
(a\Rightarrow b)\otimes a & =\bigcup\{x\otimes a\mid x\in a\Rightarrow b\}= \\
                          & =\bigcup\{\Min U(x,a')\odot a\mid x\in a'+\Max L(a,b)\}= \\
                          & =\bigcup\{\Min U(a'+y,a')\odot a\mid y\in\Max L(a,b)\}= \\
                          & =\bigcup\{(a'+y)\odot a\mid y\in\Max L(a,b)\}=\bigcup\{\{y\}\mid y\in\Max L(a,b)\}= \\
                          & =\Max L(a,b).
\end{align*}
\end{enumerate}
\end{proof}

The previous result can be generalized for subsets of $E$. This more general result will be used in the next section.

\begin{lemma}\label{lem8}
Let $(E,+,{}',0,1)$ be an effect algebra satisfying both the {\rm ACC} and the {\rm DCC} and $A,B,C\in\mathcal P_+A$. Then $A\otimes B\sqsubseteq C$ is equivalent to $A\sqsubseteq B\Rightarrow C$ {\rm(}{\em adjointness}{\rm)}
\end{lemma}

\begin{proof}
First assume $A\otimes B\sqsubseteq C$. Then there exist $a\in A$, $b\in B$, $c\in C$ and $d\in a\otimes b$ with $d\leq c$. Hence $a\otimes b\sqsubseteq c$. By adjointness this implies $a\sqsubseteq b\Rightarrow c$, i.e.\ there exists some $e\in b\Rightarrow c$ with $a\leq e$. Since $a\in A$ and $e\in B\Rightarrow C$ we conclude $A\sqsubseteq B\Rightarrow C$. Conversely, assume $A\sqsubseteq B\Rightarrow C$. Then there exist $a\in A$, $b\in B$, $c\in C$ and $e\in b\Rightarrow c$ with $a\leq e$. Hence $a\sqsubseteq b\Rightarrow c$. By adjointness this implies $a\otimes b\sqsubseteq c$, i.e.\ there exists some $f\in a\otimes b$ with $f\leq c$. Since $f\in A\otimes B$ and $c\in C$ we conclude $A\otimes B\sqsubseteq C$.
\end{proof}

The following assertion will be useful in the sequel.

\begin{lemma}\label{lem6}
Let $(E,+,{}',0,1)$ be an effect algebra satisfying both the {\rm ACC} and the {\rm DCC} and $p,q\in E^T$. Then
\begin{enumerate}[{\rm(i)}]
\item $p\leq q\Rightarrow(p\otimes q)$,
\item $(p\Rightarrow q)\otimes p\leq q$.
\end{enumerate}
\end{lemma}

\begin{proof}
\
\begin{enumerate}[(i)]
\item Since $q'\leq\Min U(p,q')$ and since $r\in\Min U(p,q')\odot q$ implies $r\leq q$ and hence $\Max L(q,r)=r$, we have
\begin{align*}
p & \leq\Min U(p,q')=q'+\big(\Min U(p,q')\odot q\big)=\bigcup\{q'+r\mid r\in\Min U(p,q')\odot q\}= \\
  & =\bigcup\{q'+\Max L(q,r)\mid r\in\Min U(p,q')\odot q\}= \\
	& =\bigcup\{q\Rightarrow r\mid r\in\Min U(p,q')\odot q\}=q\Rightarrow\big(\Min U(p,q')\odot q\big)=q\Rightarrow(p\otimes q)
\end{align*}
according to Lemma~\ref{lem3}
\item Since $r\in p'+\Max L(p,q)$ implies $p'\leq r$ and hence $\Min U(r,p')=r$, and since $\Max L(p,q)\leq p$ we have
\begin{align*}
(p\Rightarrow q)\otimes p & =\big(p'+\Max L(p,q)\big)\otimes p=\bigcup\{r\otimes p\mid r\in p'+\Max L(p,q)\}= \\
                          & =\bigcup\{\Min U(r,p')\odot p\mid r\in p'+\Max L(p,q)\}= \\
                          & =\bigcup\{r\odot p\mid r\in p'+\Max L(p,q)\}=\big(p'+\Max L(p,q)\big)\odot p= \\
												  & =\Max L(p,q)\leq q
\end{align*}
according to Lemma~\ref{lem3}.
\end{enumerate}
\end{proof}

\section{Tense operators}

For the theory of tense operators in algebraic structures see the monograph \cite{CP15a}.

Usually, the {\em tense operators} $P$, $F$, $H$, $G$ are considered. Their meaning is as follows:
\begin{align*}
P & \ldots\text{``It has at some time been the case that''}, \\
F & \ldots\text{``It will at some time be the case that''}, \\
H & \ldots\text{``It has always been the case that''}, \\
G & \ldots\text{``It will always be the case that''}.
\end{align*}
Assume that an effect algebra $(E,+,{}',0,1)$ is given as well as a non-empty time set $T$. We will consider the elements $p$ of $E^T$ as time-depending events, i.e.\ for $t\in T$, the symbol $p(t)$ denotes the value of the event $p$ at time $t$. It is easy to see that the operators $P$ and $F$ are in fact existential quantifiers over the past and future segment of $T$, respectively, similarly $H$ and $G$ are universal quantifiers over the corresponding segments. From this it is evident that for each $t\in T$ and every $p\in E^T$ we have
\begin{align*}
p(t)\leq P(p)(t) & \text{ and }p(t)\leq F(p)(t), \\
H(p)(t)\leq p(t) & \text{ and }G(p)(t)\leq p(t),
\end{align*}
shortly
\[
H(p)\leq p\leq P(p)\text{ and }G(p)\leq p\leq F(p).
\]
In order to be able to distinguish between ``past'' and ``future'' with respect to the time set $T$, we consider a so-called time-preference relation, i.e.\ a non-empty binary relation $R$ on $T$, and the couple $(T,R)$ will be called a time frame. For $s,t\in T$ with $s\mathrel Rt$ we say that ``$s$ is before $t$'' or ``$t$ is after $s$''.

For our purposes in this paper we will consider only so-called {\em serial relations} (see \cite{CP15a}), i.e.\ relations $R$ such that for each $s\in T$ there exist some $r,t\in T$ with $r\mathrel Rs$ and $s\mathrel Rt$.

Of course, if $R$ is reflexive then it is serial. Usually, $R$ is considered to be a partial order relation or a quasi-order.

In the following let $\mathbf E=(E,+,{}',0,1)$ be an effect algebra satisfying both the ACC and the DCC. We consider the set $E^T$ of all time-depending propositions on $\mathbf E$. We define the tense operators $P$, $F$, $H$ and $G$ on $\mathbf E$ to be the following mappings from $\mathcal P_+(E^T)$ to $(\mathcal P_+E)^T$:
\begin{align*}
P(B)(s) & :=\Min U\big(\{q(t)\mid q\in B\text{ and }t\mathrel Rs\}\big), \\
F(B)(s) & :=\Min U\big(\{q(t)\mid q\in B\text{ and }s\mathrel Rt\}\big), \\
H(B)(s) & :=\Max L\big(\{q(t)\mid q\in B\text{ and }t\mathrel Rs\}\big), \\
G(B)(s) & :=\Max L\big(\{q(t)\mid q\in B\text{ and }s\mathrel Rt\}\big)
\end{align*}
for all $B\in\mathcal P_+(E^T)$ and all $s\in T$. The {\em operators} defined in this way are said to be {\em induced} by $(T,R)$.

The definition of $P$, $F$, $H$ and $G$ shows some duality between $P$ and $H$ and between $F$ and $G$. In the sequel we therefore formulate and prove some results only of $H$ and $G$. The corresponding results for $P$ and $F$ are dual, in particular, $P(q)=H(q')'$ and $F(q)=G(q')'$ for each $q\in E^T$.

We want to define the composition of two tense operators. For this reason we need the following so-called {\em transformation function} $\varphi\colon(\mathcal P_+A)^T\rightarrow\mathcal P_+(A^T)$ defined by
\[
\varphi(x):=\{q\in A^T\mid q(t)\in x(t)\text{ for all }t\in T\}\text{ for all }x\in(\mathcal P_+A)^T. 
\]
We now define
\[
X*Y:=X\circ\varphi\circ Y
\]
for all tense operators $X,Y\in\{P,F,H,G\}$.

\begin{example}\label{ex2}
Consider the effect algebra from Example~\ref{ex1} with time frame $(T,R):=(\{1,2,3\},\leq)$ and the propositions $p$ and $q$ defined by
\[
\begin{array}{r|l|l|l}
   t & 1  & 2  & 3 \\
\hline
p(t) & a' & c' & a' \\
q(t) & b' & b' & c'
\end{array}
\]
Then we have
\[
\begin{array}{r|l|l|l}
                                    t & 1         & 2         & 3 \\
\hline
                                 p(t) & a'        & c'        & a' \\
                                 q(t) & b'        & b'        & c' \\
                      (p\otimes q)(t) & c         & a         & b \\
                      (q\otimes p)(t) & c         & a         & b \\
                  (p\Rightarrow q)(t) & \{b',c'\} & \{a',b'\} & c' \\
                  (q\Rightarrow p)(t) & \{d,a'\}  & \{d,c'\}  & a' \\
                              G(p)(t) & b         & b         & a' \\
                              G(q)(t) & \{a,b\}   & \{a,b\}   & c' \\
        \big(G(p)\otimes G(q)\big)(t) & 0         & 0         & b \\
        \big(G(q)\otimes G(p)\big)(t) & 0         & 0         & b \\
    \big(G(p)\Rightarrow G(q)\big)(t) & \{b',1\}  & \{b,1\}   & c' \\
    \big(G(q)\Rightarrow G(p)\big)(t) & \{a',1\}  & \{a',1\}  & a' \\
    G\big(\varphi(p\otimes q)\big)(t) & 0         & 0         & b \\
    G\big(\varphi(q\otimes p)\big)(t) & 0         & 0         & b \\
G\big(\varphi(p\Rightarrow q)\big)(t) & b         & b         & c' \\
G\big(\varphi(q\Rightarrow p)\big)(t) & b         & b         & a' \\
                              H(p)(t) & a'        & b         & b \\
                              H(q)(t) & b'        & b'        & \{a,b\} \\
        \big(H(p)\otimes H(q)\big)(t) & c         & 0         & 0 \\
        \big(H(q)\otimes H(p)\big)(t) & c         & 0         & 0 \\
    \big(H(p)\Rightarrow H(q)\big)(t) & \{b',c'\} & 1         & \{b',1\} \\
    \big(H(q)\Rightarrow H(p)\big)(t) & \{d,a'\}  & d         & \{a',1\} \\
      H\big(\varphi(p\odot q)\big)(t) & c         & 0         & 0 \\
      H\big(\varphi(q\odot p)\big)(t) & c         & 0         & 0 \\
H\big(\varphi(p\Rightarrow q)\big)(t) & \{a,b\}   & b         & b \\
H\big(\varphi(q\Rightarrow p)\big)(t) & b         & b         & b
\end{array}
\]
\end{example}

\begin{lemma}\label{lem7}
{\rm(}cf.\ {\rm\cite{CL1})} If $(E,+,{}',0,1)$ is an effect algebra, $T$ is a time set, $p\in A^T$ and $x,y\in(\mathcal P_+A)^T$ then
\begin{enumerate}[{\rm(i)}]
\item $\varphi$ is injective,
\item if $z(t):=\{p(t)\}$ for all $t\in T$ then $\varphi(z)=\{p\}$,
\item we have $x\leq y$ if and only if $\varphi(x)\leq\varphi(y)$.
\end{enumerate}
\end{lemma}

For every $q\in E^T$ let $q'\in E^T$ be defined by $q'(t):=\big(q(t)\big)'$ for all $t\in T$. For every $B\in\mathcal P_+(E^T)$ let $B'\in\mathcal P_+(E^T)$ be defined by $B':=\{q'\mid q\in B\}$. For every $x\in(\mathcal P_+E)^T$ let $x'\in(\mathcal P_+E)^T$ be defined by $x'(t):=\big(x(t)\big)'$ for all $t\in T$.

\begin{proposition}\label{prop1}
{\rm(}cf.\ {\rm\cite{CL1})} If $(E,+,{}',0,1)$ is an effect algebra satisfying both the {\rm ACC} and the {\rm DCC}, $(T,R)$ a time frame, $x\in(\mathcal P_+E)^T$, $A,B\in\mathcal P_+(A^T)$ with $A\leq B$ and $s\in T$ and $\varphi$ denotes the transformation function then
\begin{enumerate}[{\rm(i)}]
\item
\begin{align*}
P\big(\varphi(x)\big)(s) & =\Min U\big(\bigcup\{x(t)\mid t\mathrel Rs\}\big), \\
F\big(\varphi(x)\big)(s) & =\Min U\big(\bigcup\{x(t)\mid s\mathrel Rt\}\big), \\
H\big(\varphi(x)\big)(s) & =\Max L\big(\bigcup\{x(t)\mid t\mathrel Rs\}\big), \\
G\big(\varphi(x)\big)(s) & =\Max L\big(\bigcup\{x(t)\mid s\mathrel Rt\}\big),
\end{align*}
\item $H(A)=P(A')'$ and $G(A)=F(A')'$,
\item $P(A)\leq_2P(B)$, $F(A)\leq_2F(B)$, $H(A)\leq_1H(B)$ and $G(A)\leq_1G(B)$,
\item $H(A)\leq P(A)$ and $G(A)\leq F(A)$.
\end{enumerate}
\end{proposition}

\begin{remark}
In accordance with Theorem~\ref{th3} one can see that in Example~\ref{ex2} we have
\begin{align*}
                  G(p)\otimes G(q) & =G\big(\varphi(p\otimes q)\big), \\
                  G(q)\otimes G(p) & =G\big(\varphi(q\otimes p)\big), \\
                  H(p)\otimes H(q) & =H\big(\varphi(p\otimes q)\big), \\
                  H(q)\otimes H(p) & =H\big(\varphi(q\otimes p)\big), \\
G\big(\varphi(p\Rightarrow q)\big) & \leq G(p)\Rightarrow G(q), \\
G\big(\varphi(q\Rightarrow p)\big) & \leq G(q)\Rightarrow G(p), \\
H\big(\varphi(p\Rightarrow q)\big) & \leq H(p)\Rightarrow H(q), \\
H\big(\varphi(q\Rightarrow p)\big) & \leq H(q)\Rightarrow H(p).
\end{align*}
\end{remark}

The following concept was defined in \cite{CP15a}. Its specification for effect algebras is as follows.

A {\em dynamic effect algebra} is an effect algebra $(E,+,{}',0,1)$ together with two mappings $H$ and $G$ from $E^T$ to $(\mathcal P_+E)^T$ such that for all $p,q\in E^T$ the following holds:
\begin{enumerate}[(T1)]
\item $H(1)=G(1)=1$,
\item if $p\leq q$ then $H(p)\leq_1H(q)$ and $G(p)\leq_1G(q)$,
\item if $p+q$ is defined then so are $H(p)+H(q)$ and $G(p)+G(q)$ and
\begin{align*}
H(p)+H(q) & \leq_1H(p+q), \\
G(p)+G(q) & \leq_1G(p+q),
\end{align*}
\item $p\leq_1(G*P)(p)$ and $p\leq_1(H*F)(p)$ where $P(p)=H(p')'$ and $F(p)=G(p')'$.
\end{enumerate}

\begin{remark}\label{rem1}
If $(A,\leq)$ is a poset satisfying both the {\rm ACC} and the {\rm DCC} and $B,C$ are subsets of $A$ with $C\neq\emptyset$ then $B\subseteq C$ implies both $B\leq_1\Max C$ and $\Min C\leq_2B$.
\end{remark}

In what follows we show that an effect algebra with tense operators as defined above is really a dynamic one.

\begin{theorem}
Let $\mathbf E=(E,+,{}',0,1)$ be an effect algebra satisfying both the {\rm ACC} and the {\rm DCC} and $R$ a serial binary relation on $T$ and let $P$, $F$, $H$ and $G$ denote the tense operators induced by $(T,R)$. Then $\mathbf E$ together with $H$ and $G$ forms a dynamic effect algebra.
\end{theorem}

\begin{proof}
We prove only ``half'' of the statements. The rest follows in an analogous way. Because of duality reasons we have $P(p)=H(p')'$ and $F(p)=G(p')'$. Let $p,q\in E^T$ and $s\in T$.
\begin{enumerate}[(T1)]
\item $H(p)(s)=\Max L(\{1\mid t\mathrel Rs\})=1$.
\item If $p\leq q$ then
\[
H(p)(s)=\Max L\big(\{p(t)\mid t\mathrel Rs\}\big)\subseteq L\big(\{p(t)\mid t\mathrel Rs\}\big)\subseteq L\big(\{q(t)\mid t\mathrel Rs\}\big)
\]
and hence
\[
H(p)(s)\leq_1\Max L\big(\{q(t)\mid t\mathrel Rs\}\big)=H(q)(s)
\]
according to Remark~\ref{rem1}.
\item Assume $p+q$ to be defined. Then $p\leq q'$. Let
\begin{align*}
a & \in H(p)(s)=\Max L\big(\{p(t)\mid t\mathrel Rs\}\big), \\
b & \in H(q)(s)=\Max L\big(\{q(t)\mid t\mathrel Rs\}\big)
\end{align*}
and $u\in T$ with $u\mathrel Rs$. Then $a\leq p(u)$ and $b\leq q(u)$ and hence $a\leq p(u)\leq q'(u)\leq b'$, i.e. $a+b$ is defined. This shows that $H(p)+H(q)$ is defined. Because of Lemma~\ref{lem3} we have
\[
H(p)+H(q)\subseteq L\big(\{p(t)+q(t)\mid t\mathrel Rs\}\big)
\]
which by Remark~\ref{rem1} implies $H(p)+H(q)\leq_1H(p+q)$.
\item Since $P(p)=H(p')'$ we have
\begin{align*}
              P(p)(s) & =\Min U(\{p(t)\mid t\mathrel Rs\}), \\
\varphi\big(P(p)\big) & =\{r\in A^T\mid r(u)\in P(p)(u)\text{ for all }u\in T\}= \\
                      & =\{r\in A^T\mid r(u)\in \Min U(\{p(t)\mid t\mathrel Ru\})\text{ for all }u\in T\}, \\
          (G*P)(p)(s) & =G\Big(\varphi\big(P(p)\big)\Big)(s)=\Max L(\{r(v)\mid r\in\varphi\big(P(p)\big)\text{ and }s\mathrel Rv\})= \\
                      & =\Max L\big(\bigcup\{\Min U(\{p(t)\mid t\mathrel Rv\})\mid s\mathrel Rv\}\big), \\
                 p(s) & \in L\big(\bigcup\{U(\{p(t)\mid t\mathrel Rv\})\mid s\mathrel Rv\}\big)\subseteq \\
								      & \subseteq L\big(\bigcup\{\Min U(\{p(t)\mid t\mathrel Rv\})\mid s\mathrel Rv\}\big)
\end{align*}
and hence $p(s)\leq_1(G*P)(p)(s)$ according to Remark~\ref{rem1}. The last but one line can be seen as follows: If $a\in\bigcup\{U(\{p(t)\mid t\mathrel Rv\})\mid s\mathrel Rv\}$ then there exists some $v\in T$ with $s\mathrel Rv$ and $a\in U(\{p(t)\mid t\mathrel Rv\})$. But then $p(s)\leq a$.
\end{enumerate}
\end{proof}

There is the question whether the tense operators defined above are compatible with the logical connectives defined in the previous section. We are going to show that they are compatible with $\otimes$ if and only if they are compatible with $\Rightarrow$.

\begin{theorem}\label{th3}
Let $(E,+,{}',0,1)$ be an effect algebra satisfying both the {\rm ACC} and the {\rm DCC}, $(T,R)$ a time frame. Then the following holds:
\begin{enumerate}[{\rm(i)}]
\item Let $X,Y\in\{P,F,H,G\}$ and $Z\in\{H,G\}$ and assume
\[
X\big(\varphi(x)\big)\otimes Y(q)\leq_1Z\big(\varphi(x\otimes q)\big)\text{ for all }x\in(\mathcal P_+E)^T\text{ and all }q\in E^T.
\]
Then
\[
X\big(\varphi(p\Rightarrow q)\big)\sqsubseteq Y(p)\Rightarrow Z(q)\text{ for all }p,q\in E^T.
\]
\item Let $X\in\{H,G\}$ and $Y,Z\in\{P,F,H,G\}$ and assume
\[
X\big(\varphi(p\Rightarrow x)\big)\leq_1Y(p)\Rightarrow Z\big(\varphi(x)\big)\text{ for all }p\in E^T\text{ and all }x\in(\mathcal P_+E)^T.
\]
Then
\[
X(p)\otimes Y(q)\sqsubseteq Z\big(\varphi(p\otimes q)\big)\text{ for all }p,q\in E^T.
\]
\end{enumerate}
\end{theorem}

\begin{proof}
Let $p,q\in E^T$ and $x\in(\mathcal P_+E)^T$.
\begin{enumerate}[(i)]
\item Because of divisibility we have
\[
X\big(\varphi(p\Rightarrow q)\big)\otimes Y(p)\leq_1Z\Big(\varphi\big((p\Rightarrow q)\otimes p\big)\Big)=Z\Big(\varphi\big(\Max L(p,q)\big)\Big).
\]
Now $\Max L(p,q)\leq q$ implies $\varphi\big(\Max L(p,q)\big)\leq q$ according to Lemma~\ref{lem7} whence
\[
X\big(\varphi(p\Rightarrow q)\big)\otimes Y(p)\leq_1Z\Big(\varphi\big(\Max L(p,q)\big)\Big)\leq_1Z(q)
\]
according to (iii) of Proposition~\ref{prop1}. Hence $X\big(\varphi(p\Rightarrow q)\big)\otimes Y(p)\sqsubseteq Z(q)$. Adjointness and Lemma~\ref{lem8} yields
\[
X\big(\varphi(p\Rightarrow q)\big)\sqsubseteq Y(p)\Rightarrow Z(q).
\]
\item
Because of (i) of Lemma~\ref{lem6} and (iii) of Proposition~\ref{prop1} we obtain
\[
X(p)\leq_1X\Big(\varphi\big(q\Rightarrow(p\otimes q)\big)\Big)\leq_1Y(q)\Rightarrow Z\big(\varphi(p\otimes q)\big)
\]
and therefore
\[
X(p)\sqsubseteq Y(q)\Rightarrow Z\big(\varphi(p\otimes q)\big).
\]
Adjointness and Lemma~\ref{lem8} yields
\[
X(p)\otimes Y(q)\sqsubseteq Z\big(\varphi(p\otimes q)\big).
\]
\end{enumerate}
\end{proof}

\section{Constructions of a suitable preference relation}

As shown in the previous section, if a time frame $(T,R)$ is given then one can define the tense operators $P$, $F$, $H$ and $G$ on the logic derived from an effect algebra such that we obtain a dynamic effect algebra and the logical connectives $\Rightarrow$ arrow and $\otimes$ satisfy the required properties. Hence, considering tense logic based on an effect algebra, tense operators as derived here can serve for its analysis.  However, there is the question whether also conversely, when given a time set $T$ and and tense operators in the logic derived in Section~4, one can find a time preference relation such that the induced tense operators coincide with the given ones. In the following theorem we show that this is really possible since from the given tense operators we can construct a binary relation $R^*$ on $T$ such that the tense operators induced by the time frame $(T,R^*)$ are comparable with the given tense operators with respect to the quasiorder relations $\leq_1$ and $\leq_2$, respectively. Moreover, if $P$, $F$, $H$ and $G$ are induced by some time frame, then the new constructed tense operators are even equivalent to the given ones via the equivalence relations $\approx_1$ and $\approx_2$.

In the following for given effect algebra $\mathbf E=(E,+,{}',0,1)$, given time set $T$ and given tense operators $P$, $F$, $H$ and $G$ on $\mathbf E$ we call the set $R^*$ of all $(s,t)\in T^2$ satisfying
\[
H(p)(t)\leq p(s)\leq P(p)(t)\text{ and }G(p)(s)\leq p(t)\leq F(p)(s)
\]
for all $p\in E^T$ the {\em relation induced by the tense operators $P$, $F$, $H$ and $G$}.

\begin{theorem}\label{th4}
Let $\mathbf E=(E,+,{}',0,1)$ be an effect algebra satisfying both the {\rm ACC} and the {\rm DCC}, $T$ a time set and $P$, $F$, $H$ and $G$ tense operators on $\mathbf E$. Further, let $R^*$ denote the relation induced by $P$, $F$, $H$ and $G$ and $P^*$, $F^*$, $H^*$ and $G^*$ the tense operators on $\mathbf E$ induced by $(T,R^*)$. Then
\[
P^*\leq_2P, F^*\leq_2F, H\leq_1H^*\text{ and }G\leq_1G^*.
\]
\end{theorem}

\begin{proof}
Observe that $R^*$ is reflexive and hence serial. Let $q\in E^T$ and $s\in T$. Then we have 
\begin{align*}
   q(t) & \leq P(q)(s)\text{ for all }t\in T\text{ with }t\mathrel{R^*}s, \\
   q(t) & \leq F(q)(s)\text{ for all }t\in T\text{ with }s\mathrel{R^*}t, \\
H(q)(s) & \leq q(t)\text{ for all }t\in T\text{ with }t\mathrel{R^*}s, \\
G(q)(s) & \leq q(t)\text{ for all }t\in T\text{ with }s\mathrel{R^*}t
\end{align*}
and hence
\begin{align*}
P(q)(s) & \subseteq U\big(\{q(t)\mid t\mathrel{R^*}s\}\big), \\
F(q)(s) & \subseteq U\big(\{q(t)\mid s\mathrel{R^*}t\}\big), \\
H(q)(s) & \subseteq L\big(\{q(t)\mid t\mathrel{R^*}s\}\big), \\
G(q)(s) & \subseteq L\big(\{q(t)\mid s\mathrel{R^*}t\}\big)
\end{align*}
whence
\begin{align*}
P^*(q)(s) & =\Min U\big(\{q(t)\mid t\mathrel{R^*}s\}\big)\leq_2P(q)(s), \\
F^*(q)(s) & =\Min U\big(\{q(t)\mid s\mathrel{R^*}t\}\big)\leq_2F(q)(s), \\
  H(q)(s) & \leq_1\Max L\big(\{q(t)\mid t\mathrel{R^*}s\}\big)=H^*(q)(s), \\
  G(q)(s) & \leq_1\Max L\big(\{q(t)\mid s\mathrel{R^*}t\}\big)=G^*(q)(s)
\end{align*}
according to Remark~\ref{rem1}.
\end{proof}

The reason why we obtain only inequalities between the operators $P$, $F$, $H$, $G$ and $P^*$, $F^*$, $H^*$, $G^*$ is that the given ones may be rather ``exotic'' since it is not said how they are created. Hence one may ask how the mentioned operators will be related in the case that the given operators $P$, $F$, $H$ and $G$ are induced by some time frame $(T,R)$ sharing the same time set $T$. The following result shows that in this case we really obtain equivalences.

\begin{theorem}\label{th5}
Let $\mathbf E=(E,+,{}',0,1)$ be an effect algebra satisfying both the {\rm ACC} and the {\rm DCC}, $(T,R)$ a time frame and $P$, $F$, $H$ and $G$ the tense operators on $\mathbf E$ induced by $(T,R)$. Further, let $R^*$ denote the relation induced by $P$, $F$, $H$ and $G$ and $P^*$, $F^*$, $H^*$ and $G^*$ the tense operators on $\mathbf E$ induced by $(T,R^*)$. Then $R\subseteq R^*$ and
\[
P^*\approx_2P, F^*\approx_2F, H^*\approx_1H\text{ and }G^*\approx_2G.
\]
\end{theorem}

\begin{proof}
Let $q\in E^T$ and $s\in T$. If $t\in T$ and $s\mathrel Rt$ then
\begin{align*}
P(q)(t) & =\Min U(\{q(u)\mid u\mathrel Rt\})\geq q(s), \\
F(q)(s) & =\Min U(\{q(u)\mid s\mathrel Ru\})\geq q(t), \\
H(q)(t) & =\Max L(\{q(u)\mid u\mathrel Rt\})\leq q(s), \\
G(q)(s) & =\Max L(\{q(u)\mid s\mathrel Ru\})\leq q(t).
\end{align*}
This shows $R\subseteq R^*$. Now
\begin{align*}
P^*(q)(s) & =\Min U(\{q(t)\mid t\mathrel{R^*}s\})\subseteq U(\{q(t)\mid t\mathrel{R^*}s\})\subseteq U(\{q(t)\mid t\mathrel Rs\}), \\
F^*(q)(s) & =\Min U(\{q(t)\mid s\mathrel{R^*}t\})\subseteq U(\{q(t)\mid s\mathrel{R^*}t\})\subseteq U(\{q(t)\mid s\mathrel Rt\}), \\
H^*(q)(s) & =\Max L(\{q(t)\mid t\mathrel{R^*}s\})\subseteq L(\{q(t)\mid t\mathrel{R^*}s\})\subseteq L(\{q(t)\mid t\mathrel Rs\}), \\
G^*(q)(s) & =\Max L(\{q(t)\mid s\mathrel{R^*}t\})\subseteq L(\{q(t)\mid s\mathrel{R^*}t\})\subseteq L(\{q(t)\mid s\mathrel Rt\})
\end{align*}
and hence
\begin{align*}
  P(q)(s) & =\Min U(\{q(t)\mid t\mathrel Rs\})\leq_2P^*(q)(s), \\
  F(q)(s) & =\Min U(\{q(t)\mid s\mathrel Rt\})\leq_2F^*(q)(s), \\
H^*(q)(s) & \leq_1\Max L(\{q(t)\mid t\mathrel Rs\})=H(q)(s), \\
G^*(q)(s) & \leq_1\Max L(\{q(t)\mid s\mathrel Rt\})=G(q)(s)
\end{align*}
according to Remark~\ref{rem1}. This shows
\begin{align*}
  P & \leq_2P^*, \\
  F & \leq_2F^*, \\
H^* & \leq_1H, \\
G^* & \leq_1G.
\end{align*}
The rest follows from Theorem~\ref{th4}.
\end{proof}

We can ask whether there exists another construction of the time preference relation $R$ on the time set $T$ such that we obtain equivalences between the given tense operators and those induced by $(T,R)$ similarly as in Theorem~\ref{th5}. We will show that this is possible provided the time set $T$ is extended in the direction of the past as well as in the direction of the future and provided that the events are extended accordingly. Our construction is as follows.

First we introduce some notation.

Let $\mathbf E=(E,+,{}',0,1)$ be an effect algebra, $T$ a time set and $P$, $F$, $H$ and $G$ tense operators on $\mathbf E$. Put
\[
\bar T:=(T\times\{1\})\cup T\cup(T\times\{2\}),
\]
let $R^*$ denote the relation induced by $P$, $F$, $H$ and $G$
and put
\[
\bar R:=\{\big((t,1),t\big)\mid t\in T\}\cup R^*\cup\{\big(t,(t,2)\big)\mid t\in T\}.
\]
Moreover, let $\bar P$, $\bar F$, $\bar H$ and $\bar G$ denote the tense operators induced by $(\bar T,\bar R)$. Finally, for each $p\in E^T$ let $\bar p,\hat p\in E^{\bar T}$ such that
\begin{align*}
\bar p\big((t,1)\big) & \in P(p)(t), \\
            \bar p(t) & :=p(t), \\
\bar p\big((t,2)\big) & \in F(p)(t),
\end{align*}
\begin{align*}
\hat p\big((t,1)\big) & \in H(p)(t), \\
            \hat p(t) & :=p(t), \\
\hat p\big((t,2)\big) & \in G(p)(t)
\end{align*}
for all $t\in T$.

\begin{theorem}
Let $\mathbf E=(E,+,{}',0,1)$ be an effect algebra, $T$ a time set, $P$, $F$, $H$ and $G$ tense operators and $\bar T$, $R^*$, $\bar R$, $\bar P$, $\bar F$, $\bar H$, $\bar G$, $\bar p$ and $\hat p$ defined as before. Then
\[
\bar P(\bar p)|T\subseteq P(p), \bar F(\bar p)|T\subseteq F(p), \bar H(\hat p)|T\subseteq H(p)\text{ and }\bar G(\hat p)|T\subseteq G(p)
\]
for all $p\in E^T$. If $P(p)(t)$, $F(p)(t)$, $H(p)(t)$ and $G(p)(t)$ are singletons for all $p\in E^T$ and all $t\in T$ then
\[
\bar P(\bar p)|T=P(p), \bar F(\bar p)|T=F(p), \bar H(\hat p)|T=H(p)\text{ and }\bar G(\hat p)|T=G(p)
\]
for all $p\in E^T$.
\end{theorem}

\begin{proof}
Let $q\in E^T$ and $s\in T$. Then $H(q)(s)\leq q(t)\leq P(q)(s)$ and hence $\hat q\big((s,1)\big)\leq q(t)\leq\bar q\big((s,1)\big)$ for all $t\in T$ with $t\mathrel{R^*}s$ and $G(q)(s)\leq q(t)\leq F(q)(s)$ and hence $\hat q\big((s,2)\big)\leq q(t)\leq\bar q\big((s,2)\big)$ for all $t\in T$ with $s\mathrel{R^*}t$ and therefore
\begin{align*}
\bar P(\bar q)(s) & =\Min U\big(\{\bar q(\bar t)\mid\bar t\mathrel{\bar R}s\}=\Min U\Big(\{\bar q\big((s,1)\big)\}\cup\{\bar q(t)\mid t\mathrel{R^*}s\}\Big)= \\
                  & =\Min U\Big(\{\bar q\big((s,1)\big)\}\cup\{q(t)\mid t\mathrel{R^*}s\}\Big)=\{\bar q\big((s,1)\big)\}\subseteq P(q)(s), \\
\bar F(\bar q)(s) & =\Min U\big(\{\bar q(\bar t)\mid s\mathrel{\bar R}\bar t\}=\Min U\Big(\{\bar q(t)\mid s\mathrel{R^*}t\}\cup\{\bar q\big((s,2)\big)\}\Big)= \\
                  & =\Min U\Big(\{q(t)\mid s\mathrel{R^*}t\}\cup\{\bar q\big((s,2)\big)\}\Big)=\{\bar q\big((s,2)\big)\}\subseteq F(q)(s), \\
\bar H(\hat q)(s) & =\Max L\big(\{\hat q(\bar t)\mid\bar t\mathrel{\bar R}s\}=\Max L\Big(\{\hat q\big((s,1)\big)\}\cup\{\hat q(t)\mid t\mathrel{R^*}s\}\Big)= \\
                  & =\Max L\Big(\{\hat q\big((s,1)\big)\}\cup\{q(t)\mid t\mathrel{R^*}s\}\Big)=\{\hat q\big((s,1)\big)\}\subseteq H(q)(s), \\
\bar G(\hat q)(s) & =\Max L\big(\{\hat q(\bar t)\mid s\mathrel{\bar R}\bar t\}=\Max L\Big(\{\hat q(t)\mid s\mathrel{R^*}t\}\cup\{\hat q\big((s,2)\big)\}\Big)= \\
                  & =\Max L\Big(\{q(t)\mid s\mathrel{R^*}t\}\cup\{\hat q\big((s,2)\big)\}\Big)=\{\hat q\big((s,2)\big)\}\subseteq G(q)(s).
\end{align*}
If $P(q)(s)$ is a singleton then because of $\bar q\big((s,1)\big)\in P(q)(s)$ we have
\[
\bar P(\bar q)(s)=\{\bar q\big((s,1)\big)\}=P(q)(s),
\]
i.e.\ $\bar P(\bar q)|T=P(q)$. Analogous equalities hold for $F$, $H$ and $G$.
\end{proof}

\end{document}